\documentclass[english]{article}
\usepackage[T1]{fontenc}
\usepackage[latin9]{inputenc}
\usepackage{units}
\usepackage{amsmath}
\usepackage{amsthm}
\usepackage[all]{xy}

\makeatletter

\providecommand{\tabularnewline}{\\}

\theoremstyle{plain}
\newtheorem{thm}{\protect\theoremname}
  \theoremstyle{plain}
  \newtheorem{lem}[thm]{\protect\lemmaname}
  \theoremstyle{definition}
  \newtheorem{defn}[thm]{\protect\definitionname}
  \theoremstyle{remark}
  \newtheorem{rem}[thm]{\protect\remarkname}
  \theoremstyle{definition}
  \newtheorem{example}[thm]{\protect\examplename}

\makeatother

\usepackage{babel}
  \providecommand{\definitionname}{Definition}
  \providecommand{\examplename}{Example}
  \providecommand{\lemmaname}{Lemma}
  \providecommand{\remarkname}{Remark}
\providecommand{\theoremname}{Theorem}

\begin{document}

\title{Minimal Distance to Approximating Noncontextual System as a Measure
of Contextuality}

\author{Janne V. Kujala}

\date{}
\maketitle
\begin{abstract}
Let random vectors $R^{c}=\{R_{p}^{c}:p\in P_{c}\}$ represent joint
measurements of certain subsets $P_{c}\subset P$ of \emph{properties}
$p\in P$ in different \emph{contexts} $c\in C$. Such a system is
traditionally called \emph{noncontextual} if there exists a jointly
distributed set $\{Q_{p}:p\in P\}$ of random variables such that
$R^{c}$ has the same distribution as $\{Q_{p}:p\in P_{c}\}$ for
all $c\in C.$ A trivial necessary condition for noncontextuality
and a precondition for most approaches to measuring contextuality
is that the system is\emph{ consistently connected}, i.e., all $R_{p}^{c},R_{p}^{c'},\dots$
measuring the same property $p\in P$ have the same distribution.
The Contextuality-by-Default (CbD) approach allows detecting and measuring
``true'' contextuality on top of inconsistent connectedness, but
at the price of a higher computational cost. 

In this paper we propose a novel approach to measuring contextuality
that shares the generality and basic definitions of the CbD approach
and the computational benefits of the previously proposed Negative
Probability (NP) approach. The present approach differs from CbD in
that instead of considering all possible joints of the double-indexed
random variables $R_{p}^{c}$, it considers all possible approximating
\emph{single-indexed} systems $\{Q_{p}:p\in P\}$. The degree of contextuality
is defined based on the minimum possible probabilistic distance of
the actual measurements $R^{c}$ from $\{Q_{p}:p\in P_{c}\}$. We
show that the defined measure agrees with a certain measure of contextuality
of the CbD approach for all systems where each property enters in
exactly two contexts. This measure can be calculated far more efficiently
than the CbD measure and even more efficiently than the NP measure
for sufficiently large systems. The present approach can be modified
so as to agree with the NP measure of contextuality on all consistently
connected systems while extending it to inconsistently connected systems.
\end{abstract}

\section{Introduction}

\label{sec:introduction}As a basic example, consider a so called
Alice--Bob system with properties $a_{1}$, $a_{2}$ for Alice and
$b_{1},$ $b_{2}$ for Bob measured in pairs $(a_{i},b_{j})$ by the
bivariate $\pm1$-valued random variables $(A_{ij},B_{ij})$, $i,j\in\{1,2\}$.
Here $i$ is called Alice's setting and $j$ is called Bob's setting. 

Denoting by $\langle X\rangle$ the expectation of a random variable
$X$, suppose that over several trials with different settings $(i,j)$,
one observes perfectly uniform marginals $\langle A_{ij}\rangle=\langle B_{ij}\rangle=0$
for all $i,j\in\{1,2\}$ together with the product expectations (correlations)
$\langle A_{21}B_{21}\rangle=\langle A_{11}B_{11}\rangle=\langle A_{12}B_{12}\rangle=1$
and $\langle A_{22}B_{22}\rangle=-1$ for the jointly measured pairs
of properties. Assuming now that the system is \emph{noncontextual}
in the sense that it could be represented with fictitious single-indexed
jointly distributed random variables $(A_{1},A_{2},B_{1},B_{2})$
satisfying $(A_{i},B_{j})\sim(A_{ij},B_{ij})$, where $\sim$ stands
for ``has the same distribution as'', we can infer from the observed
correlations that the value of $A_{2}$ is always equal to that of
$B_{1}$ which is always equal to that of $A_{1}$ which is always
equal to that of $B_{2}$. However, the last correlation being $-1$
implies that the value of $A_{2}$ is always opposite to the measured
value of $B_{2}$. This contradiction implies that the system cannot
be noncontextual --- the single-indexed random variables representing
Alice's measurements cannot be jointly distributed with those representing
Bob's measurements. Several approaches have been developed for assessing
such contextual properties of a system. In the rest of this section,
we introduce these approaches specialized to the Alice--Bob system
(the general case will be considered later in Section~\ref{sec:general-def-results}).

In the Contextuality-by-Default (CbD) approach \cite{DK2014Scripta,DK2014PLOSconditionalization,DK2014LNCSQualified,DK2014Advances,DKL2015FooP,DKLC2015},
the jointly distributed random variables $(A_{ij},B_{ij})$ representing
measurements in a given experimental context are called \emph{bunches
}and the sets of random variables $\{A_{i1},A_{i2}\}$ and $\{B_{1j},B_{2j}\}$
for $i,j\in\{1,2\}$ representing measurements of the same property
in different contexts are called \emph{connections}. The random variables
entering a connection are said to be \emph{stochastically unrelated}
since there is no pairing-scheme to obtain a joint distribution for
the values observed in different experimental conditions. The system
is said to be \emph{consistently connected}\footnote{Consistent connectedness has also been referred to as non-signaling
and marginal selectivity in other contexts.} if $A_{i1}\sim A_{i2}$ and $B_{1j}\sim B_{2j}$ for all $i,j\in\{1,2\}$,
that is, if the distribution of $A_{ij}$ measuring Alice's property
$a_{i}$ does not depend on Bob's setting $j$ and the distribution
of $B_{ij}$ measuring Bob's property $b_{j}$ does not depend Alice's
setting $i$. If a system is not consistently connected, it is called
\emph{inconsistently connected}. An inconsistently connected system
is contextual in a trivial way since there is a direct cross-influence
from Bob's setting to the distribution of Alice's measurements or
from Alice's setting to the distribution of Bob's measurements. However,
the CbD approach allows detecting the presence of contexuality of
the type of the opening example on top of inconsistent connectedness.

A \emph{coupling} of random variables $X_{1},\dots,X_{n}$ (which
may be stochastically unrelated) is any jointly distributed random
vector $(\hat{X}_{1},\dots,\hat{X}_{n})$ that satisfies $\hat{X}_{i}\sim X_{i}$
for $i=1,\dots,n$. The coupling $(\hat{X}_{1},\dots,\hat{X}_{n})$
is said to be \emph{maximal}\footnote{\label{fn:maximal-coupling}The maximal couplings $(\hat{X}_{1},\dots,\hat{X}_{n})$
of discrete random variables $(X_{1},\dots,X_{n})$ satisfy

\[
\Pr[\hat{X}_{1}=\dots=\hat{X}_{n}=x]=\min_{i\in\{1,\dots,n\}}\Pr[X_{i}=x]
\]
for all $x$ (see Thorisson \cite{Thorisson2000}, Chapter 1, Theorem~4.2)
and so the maximum coupling probability is given by the sum of the
above probability over all $x$. The same result generalizes to densities
of arbitrary random variables (see \cite{Thorisson2000}, Chapter
3, Section 7).} if the \emph{coupling probability} $\Pr[\hat{X}_{1}=\dots=\hat{X}_{n}]$
is maximum possible given the distributions of $X_{1},\dots,X_{n}$.
In the CbD approach, one considers couplings $\{(\hat{A}_{ij},\hat{B}_{ij})\}_{i,j\in\{1,2\}}$
of the bunches $\{(A_{ij},B_{ij})\}_{i,j\in\{1,2\}}$ as well as the
subcouplings $(\hat{A}_{i1},\hat{A}_{i2})$ and $(\hat{B}_{1j},\hat{B}_{2j})$
for $i,j\in\{1,2\}$ corresponding to connections. A system is said
to be \emph{noncontextual} if there exists a coupling $\{(\hat{A}_{ij},\hat{B}_{ij})\}_{i,j\in\{1,2\}}$
of the bunches $\{(A_{ij},B_{ij})\}_{i,j\in\{1,2\}}$ such that the
subcouplings corresponding to connections are maximal.

For a consistently connected system, this definition reduces to the
existence of a coupling $\{(\hat{A}_{ij},\hat{B}_{ij})\}_{i,j\in\{1,2\}}$
of the bunches $\{(A_{ij},B_{ij})\}_{i,j\in\{1,2\}}$ such that $\hat{A}_{i1}=\hat{A}_{i2}$
and $\hat{B}_{1j}=\hat{B}_{2j}$ for all $i,j\in\{1,2\}$ (and is
therefore equivalent to the traditional understanding of contextuality
for consistently connected systems). If such a coupling cannot be
found, one can measure how far from maximality the subcouplings are
by determining the minimum possible value of 
\begin{equation}
\Delta^{\textnormal{CbD}}=\sum_{i\in\{1,2\}}\Pr[\hat{A}_{i1}\ne\hat{A}_{i2}]+\sum_{j\in\{1,2\}}\Pr[\hat{B}_{1j}\ne\hat{B}_{2j}]\label{eq:AliceBob_DeltaCbD}
\end{equation}
over all couplings $\{(\hat{A}_{ij},\hat{B}_{ij})\}_{i,j\in\{1,2\}}$
of the pairs $\{(A_{ij},B_{ij})\}_{i,j\in\{1,2\}}$. When the system
is inconsistently connected, the above measure of contextuality generalizes
to the excess of $\Delta^{\textnormal{CbD}}$ over its minimum possible
value
\begin{equation}
\Delta_{0}^{\textnormal{CbD}}=\sum_{i\in\{1,2\}}\frac{1}{2}\left|\left\langle A_{i1}\right\rangle -\left\langle A_{i2}\right\rangle \right|+\sum_{j\in\{1,2\}}\frac{1}{2}\left|\left\langle B_{1j}\right\rangle -\left\langle B_{2j}\right\rangle \right|\label{eq:AliceBob_Delta0CbD}
\end{equation}
allowed by the marginal expectations $\langle A_{ij}\rangle,\langle B_{ij}\rangle$,
$i,j\in\{1,2\}$ so that the system is noncontextual in general if
and only $\Delta^{\text{CbD}}-\Delta_{0}^{\text{CbD}}=0$. Other measures
consistent with the definition of contextuality can be used with the
CbD approach as well.

In the negative probability (NP) approach \cite{AbramskyBrandenburger2011,deBarros2014,Oas2014,Spekkens2008},
the system is said to be noncontextual if there exists a joint distribution
of fictitious random variables $(A_{1},A_{2},B_{1},B_{2})$ such that
$(A_{i},B_{j})\sim(A_{ij},B_{ij})$ for all $i,j\in\{1,2\}$. If this
is not possible, then the degree of contextuality is defined as the
minimum possible negative probability mass 
\begin{align*}
\Delta^{\textnormal{NP}} & =\sum_{a_{1},a_{2},b_{1},b_{2}\in\{-1,1\}}\min\{0,p(a_{1},a_{2},b_{1},b_{2})\}\\
 & =-1+\sum_{a_{1},a_{2},b_{1},b_{2}\in\{-1,1\}}\left|p(a_{1},a_{2},b_{1},b_{2})\right|.
\end{align*}
of a negative probability joint $p(a_{1},a_{2},b_{1},b_{2})$ (a real-valued
function summing to 1 over all $a_{1},a_{2},b_{1},b_{2}\in\{-1,1\}$)
of $(A_{1},A_{2},B_{1},B_{2})$ such that $(A_{i},B_{j})\sim(A_{ij},B_{ij})$
for all $i,j\in\{1,2\}$, that is, 
\[
\sum_{a_{3-i},b_{3-j}\in\{-1,1\}}p(a_{1},a_{2},b_{1},b_{2})=\Pr[A_{ij}=a_{i},B_{ij}=b_{j}]
\]
for all $a_{i},b_{j}\in\{-1,1\}$ and $i,j\in\{1,2\}$. Such a negative
probability joint is known to exist provided that $\langle A_{i1}\rangle=\langle A_{i2}\rangle$
and $\langle B_{1j}\rangle=\langle B_{2j}\rangle$ for all $i,j\in\{1,2\}$.

It has been shown \cite{deBarrosDzhafarovKujalaOas2015} that both
approaches give the same numerical values for the degree of contextuality
for the simplest Alice--Bob system considered here as well as for
some other simple systems. The CbD approach, however, is more general
in that it can be directly applied even when $\langle A_{i1}\rangle\ne\langle A_{i2}\rangle$
or $\langle B_{1j}\rangle\ne\langle B_{2j}\rangle$ whereas the negative
probability joint for single-indexed variables only exists when the
system is consistently connected. 

Although the CbD approach has wider applicability, the NP approach
has the benefit of being computationally far more efficient. This
is because the number of variables needed to represent the joint probabilities
of a single-indexed joint is much smaller than the number of variables
needed to represent the double-indexed coupling of the CbD approach. 

In this paper, we present a novel approach to measuring contextuality
that has the same generality and formal language as the CbD approach
while in certain specializations reaching and potentially surpassing
the computational efficiency of the NP approach. Thus, on these desiderata,
the present approach appears to be superior to either of the previous
approaches. Alternatively, the present approach can be formulated
so as to agree with either approach although the computational benefits
over the compared system will then be lost. In particular, a certain
formulation agrees with the NP approach on all consistently connected
systems while extending it to inconsistently connected systems. 

In the following sections, we first sketch the ideas using the simplest
Alice--Bob example considered above. Then, we formulate the general
theory using mathematically rigorous definitions and theorems extending
on the framework of the CbD approach. After that, we illustrate the
computational benefits using the Alice--Bob example with $m$ settings
for Alice and $n$ settings for Bob. We conclude by giving an abstract
view to the logic of the present approach and compare its similarities
and differences to CbD as well as consider their implications.

\section{Defining an approximating non-contextual system: a sketch of the
ideas}

In this section, we introduce the basic ideas using the same Alice--Bob
example as in the introduction.
\begin{lem}
\label{lem:min-pr}Given jointly distributed $\pm1$-valued random
variables $A$, $A'$, and $A''$, we have
\begin{equation}
\frac{1}{2}\left|\left\langle A'\right\rangle -\left\langle A''\right\rangle \right|\le\Pr[A'\ne A'']\le\Pr[A\ne A']+\Pr[A\ne A''].\label{eq:min-pr}
\end{equation}
Furthermore, for any given values of $\left\langle A'\right\rangle ,\left\langle A''\right\rangle \in[-1,1]$
with $\left\langle A\right\rangle $ in between, there exists jointly
distributed $(A,A',A'')$ such that both of the above inequalities
hold as equality. \end{lem}
\begin{proof}
The maximum coupling probability of $A'$ and $A''$ is given by
\begin{align*}
 & \min\left\{ \Pr[A'=1],\Pr[A''=1]\right\} +\min\left\{ \Pr[A'=-1],\Pr[A''=-1]\right\} \\
 & =\min\left\{ \frac{1}{2}+\frac{1}{2}\left\langle A'\right\rangle ,\frac{1}{2}+\frac{1}{2}\left\langle A''\right\rangle \right\} +\min\left\{ \frac{1}{2}-\frac{1}{2}\left\langle A'\right\rangle ,\frac{1}{2}-\frac{1}{2}\left\langle A''\right\rangle \right\} \\
 & =1-\frac{1}{2}\left[\max\{\left\langle A'\right\rangle ,\left\langle A''\right\rangle \}-\min\{\left\langle A'\right\rangle ,\left\langle A''\right\rangle \}\right]=1-\frac{1}{2}\left|\left\langle A'\right\rangle -\left\langle A''\right\rangle \right|,
\end{align*}
which yields the left inequality and the fact that it can hold as
an equality for some joint with the given marginal expectations. The
right inequality follows from the fact that the event $A'\ne A''$
is contained in the event $A\ne A'$ or $A\ne A''$ (since $A=A'$
and $A=A''$ imply $A'=A''$) and both inequalities can be made to
hold as an equality by letting $(A',A'')$ be a maximal coupling and
then choose $A$ to be either $A'$ or $A''$ to make equality on
the right. This yields the result for $\left\langle A\right\rangle $
equal to $\left\langle A'\right\rangle $ or $\left\langle A''\right\rangle $;
for a value of $\left\langle A\right\rangle $ between $\left\langle A'\right\rangle $
and $\left\langle A''\right\rangle $, we can then take a convex combination
of the joints of $(A,A',A'')$ in the two cases.\end{proof}
\begin{defn}
\label{def:noncontextual}We say that the Alice--Bob system is noncontextual,
if there exists jointly distributed fictitious $\pm1$-valued random
variables $(A_{1},A_{2},B_{1},B_{2})$ such that each pair $(A_{i},B_{j})$
can be coupled with the observable pair $(A_{ij},B_{ij})$ by a coupling
\begin{equation}
((\hat{A}_{i},\hat{B}_{j}),(\hat{A}_{ij},\hat{B}_{ij}))\label{eq:AB-coupling}
\end{equation}
such that over the four such couplings, the sum 
\begin{equation}
\Delta=\sum_{i,j\in\{1,2\}}\left(\Pr[\hat{A}_{ij}\ne\hat{A}_{i}]+\Pr[\hat{B}_{ij}\ne\hat{B}_{j}]\right)\label{eq:consum}
\end{equation}
equals
\begin{equation}
\Delta_{0}=\sum_{i\in\{1,2\}}\frac{1}{2}\left|\left\langle A_{i1}\right\rangle -\left\langle A_{i2}\right\rangle \right|+\sum_{j\in\{1,2\}}\frac{1}{2}\left|\left\langle B_{1j}\right\rangle -\left\langle B_{2j}\right\rangle \right|,\label{eq:minsum}
\end{equation}
which is (by Lemma~\ref{lem:min-pr})\footnote{The precise argument here is that for $i\in\{1,2\}$ and any joint
of $(\hat{A}_{i},\hat{A}_{i1},\hat{A}_{i2})$ having the marginal
expectations $\left\langle A_{i}\right\rangle ,\left\langle A_{i1}\right\rangle ,\left\langle A_{i2}\right\rangle $,
Lemma~\ref{lem:min-pr} implies that $\frac{1}{2}\left|\left\langle A_{i1}\right\rangle -\left\langle A_{i2}\right\rangle \right|\le\Pr[\hat{A}_{i1}\ne\hat{A}_{i}]+\Pr[\hat{A}_{i2}\ne\hat{A}_{i}]$
with equality for some joint and analogously for $(\hat{B}_{i},\hat{B}_{1j},\hat{B}_{2j})$
for $j\in\{1,2\}$. Given joints for which the equality holds, we
obtain a $\Delta$ equaling $\Delta_{0}$ with the couplings $(\hat{A}_{i},\hat{A}_{ij})$
and $(\hat{B}_{j},\hat{B}_{ij})$ defined as marginals of the 3-joints
with $(A_{1},A_{2},B_{1},B_{2})$ given as an arbitrary joint with
the same marginals as $\hat{A}_{1},\hat{A}_{2},\hat{B}_{1},\hat{B}_{2}$.
Conversely, any $A_{1},A_{2},B_{1},B_{2}$ together with couplings
$(\hat{A}_{i},\hat{A}_{ij})$ and $(\hat{B}_{j},\hat{B}_{ij})$ for
$i,j\in\{1,2\}$ can be used to define joints for $(\hat{A}_{i},\hat{A}_{i1},\hat{A}_{i2})$
and $(\hat{B}_{i},\hat{B}_{1j},\hat{B}_{2j})$ whose 2-marginals agree
with $(\hat{A}_{i},\hat{A}_{ij})$ and $(\hat{B}_{j},\hat{B}_{ij})$
and so it follows $\Delta\ge\Delta_{0}$ implying that no value smaller
than $\Delta_{0}$ can be obtained.} the minimum possible value of (\ref{eq:consum}) over all possible
separate couplings $(\hat{A}_{i},\hat{A}_{ij})$ and $(\hat{B}_{j},\hat{B}_{ij})$
of respectively $(A_{i},A_{ij})$ and $(B_{j},B_{ij})$ for $i,j\in\{1,2\}$
and all possible choices of jointly distributed $\pm1$-valued random
variables $(A_{1},A_{2},B_{1},B_{2})$. A measure of contextuality
is given by the minimum possible value of the difference $\Delta-\Delta_{0}$
over the possible couplings and choices of $(A_{1},A_{2},B_{1},B_{2})$.

Intuitively, the single-indexed system $(A_{1},A_{2},B_{1},B_{2})$
in the above definition approximates the double-indexed system $\{(A_{ij},B_{ij}):i,j\in\{1,2\}\}$
and the distance of the double-indexed system to the approximating
single indexed-system is given by the minimum possible value of $\Delta$
over all possible couplings $((\hat{A}_{i},\hat{B}_{j}),(\hat{A}_{ij},\hat{B}_{ij}))$
of $((A_{i},B_{j}),(A_{ij},B_{ij}))$ for all $i,j\in\{1,2\}$. If
this distance equals $\Delta_{0}$, then the error of the approximation
is the minimum possible allowed by the marginal distributions of $A_{ij}$
and $B_{ij}$ for $i,j\in\{1,2\}$ and the system is considered noncontextual
(apart from any inconsistent connectedness). \end{defn}
\begin{thm}
A system is noncontextual according to Definition~\ref{def:noncontextual}
if and only if it is noncontextual according to the standard CbD definition
as well. \end{thm}
\begin{proof}
The four couplings (\ref{eq:AB-coupling}) together with the jointly
distributed variables $(A_{1},A_{2},B_{1},B_{2})$ induce a coupling
$((\hat{A}_{1},\hat{A}_{2},\hat{B}_{1},\hat{B}_{2}),\{(\hat{A}_{ij},\hat{B}_{ij}):i,j\in\{1,2\}\})$
of all variables $((A_{1},A_{2},B_{1},B_{2}),\{(A_{ij},B_{ij}):i,j\in\{1,2\}\})$
with the pairs $(\hat{A}_{ij},\hat{B}_{ij})$ independent of each
other given $(\hat{A}_{1},\hat{A}_{2},\hat{B}_{1},\hat{B}_{2})$:
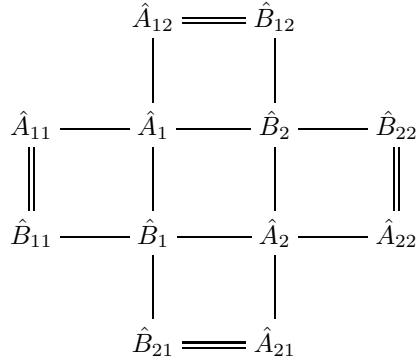
\[
\xymatrix{ & \hat{A}_{12}\ar@{=}[r] & \hat{B}_{12}\\
\hat{A}_{11} & \hat{A}_{1}\ar@{-}[l]\ar@{-}[u]\ar@{-}[r] & \hat{B}_{2}\ar@{-}[u]\ar@{-}[r]\ar@{-}[d] & \hat{B}_{22}\ar@{=}[d]\\
\hat{B}_{11}\ar@{=}[u] & \hat{B}_{1}\ar@{-}[d]\ar@{-}[l]\ar@{-}[u] & \hat{A}_{2}\ar@{-}[r]\ar@{-}[d]\ar@{-}[l] & \hat{A}_{22}\\
 & \hat{B}_{21} & \hat{A}_{21}\ar@{=}[l]
}
\]
The subcoupling $\{(\hat{A}_{ij},\hat{B}_{ij}):i,j\in\{1,2\}\}$ (leaving
out the single indexed variables) is jointly distributed and (\ref{eq:consum})
equaling (\ref{eq:minsum}) implies by Lemma~\ref{lem:min-pr} that
\begin{align}
 & \sum_{i\in\{1,2\}}\Pr[\hat{A}_{i1}\ne\hat{A}_{i2}]+\sum_{j\in\{1,2\}}\Pr[\hat{B}_{1j}\ne\hat{B}_{2j}]\nonumber \\
 & =\sum_{i\in\{1,2\}}\frac{1}{2}\left|\langle\hat{A}_{i1}\rangle-\langle\hat{A}_{i2}\rangle\right|+\sum_{j\in\{1,2\}}\frac{1}{2}\left|\langle\hat{B}_{1j}\rangle-\langle\hat{B}_{2j}\rangle\right|,\label{eq:normal_criterion}
\end{align}
which is precisely the criterion for noncontextuality according to
the standard CbD definition given by (\ref{eq:AliceBob_DeltaCbD})
and (\ref{eq:AliceBob_Delta0CbD}).

Conversely, suppose that (\ref{eq:normal_criterion}) holds for some
coupling $\{(\hat{A}_{ij},\hat{B}_{ij}):i,j\in\{1,2\}\}$ of $\{(A_{ij},B_{ij}):i,j\in\{1,2\}\}$.
Then, for all $i,j\in\{1,2\}$, we can choose $\hat{A}_{i}$ and $A_{i}$
both equal to $\hat{A}_{ik}$ for arbitrary $k\in\{1,2\}$ and $\hat{B}_{j}$
and $B_{j}$ both equal to $\hat{B}_{kj}$ for arbitrary $k\in\{1,2\}$.
It immediately follows that $((\hat{A}_{i},\hat{B}_{j}),(\hat{A}_{ij},\hat{B}_{ij}))$
is a coupling of $((A_{i},B_{j}),(A_{ij},B_{ij}))$ for all $i,j\in\{1,2\}$
and that these four couplings satisfy 
\begin{align*}
 & \sum_{i,j\in\{1,2\}}\left(\Pr[\hat{A}_{ij}\ne\hat{A}_{i}]+\Pr[\hat{B}_{ij}\ne\hat{B}_{j}]\right)\\
 & =\sum_{i\in\{1,2\}}\Pr[\hat{A}_{i1}\ne\hat{A}_{i2}]+\sum_{j\in\{1,2\}}\Pr[\hat{B}_{1j}\ne\hat{B}_{2j}]\\
 & =\sum_{i\in\{1,2\}}\frac{1}{2}\left|\langle\hat{A}_{i1}\rangle-\langle\hat{A}_{i2}\rangle\right|+\sum_{j\in\{1,2\}}\frac{1}{2}\left|\langle\hat{B}_{1j}\rangle-\langle\hat{B}_{2j}\rangle\right|
\end{align*}
so that Definition~\ref{def:noncontextual} is satisfied.
\end{proof}
It is not essential for the idea of approximating one system with
another that the approximating system is a proper noncontextual system.
We can in fact define a concept of optimal approximation by any system
that can predict consistently connected observable joint distributions:
\begin{defn}
\label{def:approximating-system}A set of bivariate random variables
$\{(A'_{ij},B'_{ij}):i,j\in\{1,2\}\}$ that is consistently connected,
i.e., satisfies $A'_{ij}\sim A'_{ij'}$ and $B'_{ij}\sim B'_{i'j}$
for all $i,i',j,j'\in\{1,2\}$, is said to \emph{approximate optimally}
the system $\{(A{}_{ij},B{}_{ij}):i,j\in\{1,2\}\}$ if there exists
couplings $((\hat{A}_{ij},\hat{B}_{ij}),(\hat{A}'_{ij},\hat{B}'_{ij}))$
for all $i,j\in\{1,2\}$ such that 
\[
\Delta'=\sum_{i,j\in\{1,2\}}\left(\Pr[\hat{A}_{ij}\ne\hat{A}'_{ij}]+\Pr[\hat{B}_{ij}\ne\hat{B}'_{ij}]\right)
\]
equals
\[
\Delta_{0}=\sum_{i\in\{1,2\}}\frac{1}{2}\left|\left\langle A_{i1}\right\rangle -\left\langle A_{i2}\right\rangle \right|+\sum_{j\in\{1,2\}}\frac{1}{2}\left|\left\langle B_{1j}\right\rangle -\left\langle B_{2j}\right\rangle \right|.
\]
This construction is illustrated by the following diagram:
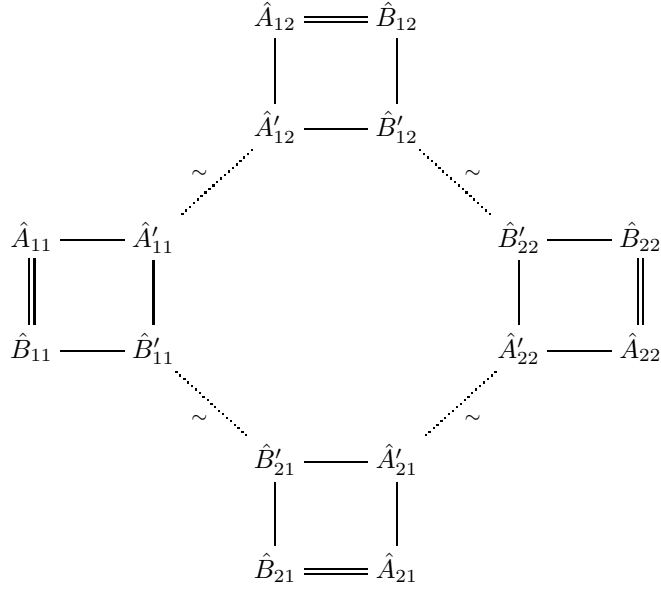
\[
\xymatrix{ &  & \hat{A}_{12}\ar@{=}[r]\ar@{-}[d] & \hat{B}_{12}\ar@{-}[d]\\
 &  & \hat{A}'_{12}\ar@{-}[r] & \hat{B}'_{12}\ar@{.}[dr]^{\sim}\\
\hat{A}_{11}\ar@{-}[r] & \hat{A}'_{11}\ar@{.}[ur]^{\sim} &  &  & \hat{B}'_{22}\ar@{-}[d] & \hat{B}_{22}\ar@{=}[d]\ar@{-}[l]\\
\hat{B}_{11}\ar@{=}[u]\ar@{-}[r] & \hat{B}'_{11}\ar@{-}[u] &  &  & \hat{A}'_{22}\ar@{.}[dl]^{\sim} & \hat{A}_{22}\ar@{-}[l]\\
 &  & \hat{B}'_{21}\ar@{.}[ul]^{\sim} & \hat{A}'_{21}\ar@{-}[l]\\
 &  & \hat{B}_{21}\ar@{-}[u] & \hat{A}_{21}\ar@{=}[l]\ar@{-}[u]
}
\]
\end{defn}
\begin{rem}
Asystem is noncontextual according to Definition~\ref{def:noncontextual}
if and only if it is approximated optimally by a system that is noncontextual
in the traditional sense. 
\end{rem}
Definition~\ref{def:approximating-system} above allows in particular
approximating by a negative probability system:
\begin{defn}
Consider all negative probability joints of $(A_{1},A_{2},B_{1},B_{2})$
having proper marginals for all $(A_{i},B_{j})$ with $\Delta'=\Delta_{0}$
where we define $(A'_{ij},B'_{ij})=(A_{i},B_{j})$. The degree of
contextuality in the system is then defined as the minimum possible
total negative probability mass among all such optimally approximating
negative probability joints.\end{defn}
\begin{example}
\label{ex:EPR}Definition~\ref{def:approximating-system} can also
be applied to a specific model predicting a consistently connected
set of jointly distributed pairs $(A'_{ij},B'_{ij})$, for example,
to the quantum model 
\[
\langle A'_{ij}B'_{ij}\rangle=-\cos(\alpha_{i}-\beta_{j}),\quad\langle A'_{ij}\rangle=\langle B'_{ij}\rangle=\nicefrac{1}{2},\quad i,j\in\{1,2\}
\]
of the Einstein--Podolsky--Rosen (EPR) experiment for photons (see,
e.g., \cite{Aspect1999}). Thus, if the observations deviate somewhat
from consistent connectedness, the above approach still allows one
to test whether the observations are as close to the prediction as
possible ignoring the contextual changes in the marginals. This allows,
for example, $\langle A_{11}\rangle$ and $\langle A_{12}\rangle$
to deviate from the predicted value $\nicefrac{1}{2}$ in some cases,
but only if they deviate to different directions.
\end{example}

\section{General definition and results}

\label{sec:general-def-results}In this section, we formalize the
ideas presented in the previous section into a general and rigorous
mathematical formulation.

Let $P$ denote the set of all properties and $C$ the set of all
contexts. For simplicity, we assume that there is a finite number
of elements in $C$ and $P$ (although most of the results should
be generalizable to countable $C$ and $P$ as well). Let $R^{c}=\{R_{p}^{c}:p\in P^{c}\}$
denote a set of jointly distributed random variables $R_{p}^{c}:\Omega_{c}\to E_{p}$
(with values in arbitrary spaces $E_{p}$) representing the observation
of properties $p\in P^{c}$ that enter in context $c\in C$. Furthermore,
let $C_{p}\subset C$ denote the set of all contexts in which the
property $p$ enters. Following the Contextuality-by-Default terminology,
the jointly observable set $R^{c}$ of random variables entering in
a given context $c$ will be called a \emph{bunch} in the following
(as will any analogously indexed set of jointly distributed random
variables).

The CbD criterion of contextuality is given by the existence of a
coupling $\{\hat{R}^{c}:c\in C\}$ of $\{R^{c}:c\in C\}$ satisfying
$\Delta^{\textnormal{CbD}}=\Delta_{0}^{\textnormal{CbD}}$, where
\begin{equation}
\Delta^{\textnormal{CbD}}=\sum_{p\in P}\Pr[\hat{R}_{p}^{c}\ne\hat{R}_{p}^{c'}\textnormal{ for some }c,c'\in C_{p}]\label{eq:Delta_CbD_general}
\end{equation}
and
\begin{equation}
\Delta_{0}^{\textnormal{CbD}}=\sum_{p\in P}\min_{\substack{\text{all couplings}\\
\{T^{c}:c\in C_{p}\}\text{ of }\{R_{p}^{c}:c\in C_{p}\}
}
}\Pr[T^{c}\ne T^{c'}\textnormal{ for some }c,c'\in C_{p}].\label{eq:Delta0_CbD_general}
\end{equation}
As discussed in Section~\ref{sec:introduction}, this condition implies
that all subcouplings $\{\hat{R}_{p}^{c}:c\in C_{p}\}$ corresponding
to a connection $\{R_{p}^{c}:c\in C_{p}\}$, $p\in P$, are maximal.
This formulation of the criterion also yields a measure of contextuality
$\Delta^{\textnormal{CbD}}-\Delta_{0}^{\textnormal{CbD}}$ to which
we will compare the measure of contextuality of the present approach.
Note however, that the CbD approach allows for other measures of contextuality
as well.

Let us then give the general definitions of the present apporach.
Given jointly distributed single-indexed random variables $\{Q_{p}:p\in P\}$,
we define the bunches $Q^{c}=\{Q_{p}:p\in P^{c}\}$ with elements
indexed as $Q_{p}^{c}=Q_{p}$. 

Given couplings $(\hat{R}^{c},\hat{Q}^{c})$ of $(R^{c},Q^{c})$ for
every $c\in C$, let us denote

\begin{equation}
\Delta=\sum_{p\in P}\sum_{c\in C_{p}}\Pr[\hat{Q}_{p}^{c}\ne\hat{R}_{p}^{c}]\label{eq:sumdiffprob}
\end{equation}
and
\begin{equation}
\Delta_{0}=\sum_{p\in P}\Delta_{p},\label{eq:Delta_0}
\end{equation}
where
\begin{equation}
\Delta_{p}=\min_{\textnormal{all choices of }Q_{p}}\sum_{c\in C_{p}}\min_{\substack{\text{all couplings }(T_{p},S_{p})\\
\textnormal{of }(Q_{p},R_{p}^{c})
}
}\Pr[T_{p}\ne S_{p}].\label{eq:Delta_p}
\end{equation}

\begin{rem}
Note that $\hat{Q}_{p}^{c}$ in (\ref{eq:sumdiffprob}) is distinct
for every $c$ (even if the original variables are determined by $Q_{p}^{c}=Q_{p}$
for all $c$). \end{rem}
\begin{defn}
\label{def:noncontextual-inside-measure}The system is \emph{noncontextual}
if there exists jointly distributed random variables $\{Q_{p}:p\in P\}$
and couplings $(\hat{R}^{c},\hat{Q}^{c})$ of $(R^{c},Q^{c})$ for
every $c\in C$ such that $\Delta=\Delta_{0}$. Furthermore, we define
the minimum possible value of the difference $\Delta-\Delta_{0}$
over all possible choices of the random variables $\{Q_{p}:p\in P\}$
and couplings $\{(\hat{R}^{c},\hat{Q}^{c}):c\in C\}$ as a measure
of contextuality.\end{defn}
\begin{rem}
\label{rem:joint-couplings}For any joint distribution of $\{Q_{p}:p\in P\}$
and any distributions of the couplings $\{(\hat{R}^{c},\hat{Q}^{c}):c\in C\}$,
we can construct all the random variables appearing in these so that
they are jointly distributed on a common probability space and satisfy
$\hat{Q}_{p}^{c}=Q_{p}$ for all $c\in C$ with the bunches $\{\hat{R}^{c}:c\in C\text{\}}$
independent of each other given $\{Q_{p}:p\in P\}$.
\end{rem}

\begin{rem}
\label{rem:componentwise-maximality}For given $Q_{p}$, the couplings
$(T_{p},S_{p})$ of $(Q_{p},R_{p}^{c})$ that minimize the sum in
(\ref{eq:Delta_p}) are precisely those that are maximal. Thus, for
given $\{Q^{c}:c\in C\}$, the couplings $(\hat{R}^{c},\hat{Q}^{c})$
of $(R^{c},Q^{c})$ that yield (\ref{eq:sumdiffprob}) equal to the
sum over $p\in P$ of the sum in (\ref{eq:Delta_p}) are precisely
those whose subcouplings $(\hat{R}_{p}^{c},\hat{Q}_{p}^{c})$ for
all $p\in P^{c}$, $c\in C$, are maximal. 
\end{rem}

\begin{defn}
\label{def:general-approximation}A set $\{Q^{c}:c\in C\}$ of random
bunches $Q^{c}=\{Q_{p}^{c}:p\in P^{c}\}$ satisfying consistent connectedness
condition (i.e., $Q_{p}^{c}\sim Q_{p}^{c'}$ for all $c,c'\in C_{p}$
and all $p\in P$) is said to \emph{approximate optimally} the set
$\{R^{c}:c\in C\}$ of random bunches $R^{c}=\{R_{p}^{c}:p\in P^{c}\}$
if there exist couplings $(\hat{R}^{c},\hat{Q}^{c})$ of $(R^{c},Q^{c})$
for every $c\in C$ such that $\Delta=\Delta_{0}$.
\end{defn}

\begin{defn}
\label{def:NP-inside-measure}An alternative measure of contextuality
is given by the minimum negative probability mass over all negative
probability joints of $\{Q_{p}:p\in P\}$ whose marginals $\{Q_{p}:p\in P^{c}\}$
approximate optimally the observed bunches $\{R^{c}:c\in C\}$. Expanding
the definitions, this means that
\begin{enumerate}
\item for every $c\in C$, the marginal distribution for the subset $Q^{c}=\{Q_{p}:p\in P^{c}\}$
has proper nonnegative probabilities, and
\item over all $c\in C$, there exists couplings $(\hat{R}^{c},\hat{Q}^{c})$
of $(R^{c},Q^{c})$ for every $c\in C$ satisfying $\Delta=\Delta_{0}$.
\end{enumerate}
\end{defn}
Next, we will list several simple theorems characterizing some of
the properties of the above definitions. 
\begin{thm}
\label{thm:equivalence-with-NP}For a consistently connected system,
the measure of Definition~\ref{def:NP-inside-measure} agrees with
the NP measure of contextuality. \end{thm}
\begin{proof}
For a consistently connected system we have $\Delta_{0}=0$ and $\Delta=\Delta_{0}$
then implies that $\hat{Q}_{p}^{c}=\hat{R}_{p}^{c}$ for all $p\in P^{c}$
and $c\in C$ and so constraint 2 of Definition~\ref{def:NP-inside-measure}
reduces to the requirement that $\{Q_{p}:p\in P^{c}\}\sim\{R_{c}^{p}:p\in P^{c}\}$
for all $c$. This is precisely the standard NP approach.
\end{proof}

\begin{thm}
\label{thm:equivalence-with-CbD}For a system where every property
$p\in P$ enters in exactly two contexts $c_{p}$ and $c'_{p}$, the
minimum possible value of $\Delta$ over all couplings $\{(\hat{R}^{c},\hat{Q}^{c})\}$
for all $c\in C$ is the same as the minimum possible value of

\begin{equation}
\Delta^{\textnormal{CbD}}=\sum_{p\in P}\Pr[\hat{R}_{p}^{c_{p}}\ne\hat{R}_{p}^{c'_{p}}]\label{eq:cbd_delta}
\end{equation}
over all couplings $\{\hat{R}^{c}:c\in C\}$ of $\{R^{c}:c\in C\}$
and $\Delta_{0}$ equals
\[
\Delta_{0}^{\textnormal{CbD}}=\sum_{p\in P}\min_{\substack{\text{all couplings}\\
(T^{c_{p}},T^{c'_{p}}\}\text{ of }\{R_{p}^{c_{p}},R_{p}^{c'_{p}}\}
}
}\Pr[T^{c_{p}}\ne T^{c'_{p}}]
\]
This implies that under these assumptions, the measure of contextuality
$\Delta-\Delta_{0}$ is equivalent to the CbD definition given by
(\ref{eq:Delta_CbD_general}) and (\ref{eq:Delta0_CbD_general}).\end{thm}
\begin{proof}
Let $\{(\hat{R}^{c},\{\hat{Q}_{p}^{c}:p\in P^{c}\})\}$ for $c\in C$
be the couplings with the minimum possible value of $\Delta$. As
per Remark~\ref{rem:joint-couplings}, we can assume all the couplings
to be jointly distributed with $\hat{Q}_{p}^{c_{p}}=\hat{Q}_{p}^{c'_{p}}$.
Then, $\hat{Q}_{p}^{c_{p}}=\hat{R}_{p}^{c_{p}}$ and $\hat{Q}_{p}^{c'_{p}}=\hat{R}_{p}^{c'_{p}}$
imply $\hat{R}_{p}^{c_{p}}=\hat{R}_{p}^{c'_{p}}$, which yields (by
the same argument as in the proof of Lemma~\ref{lem:min-pr}) 
\[
\Pr[\hat{R}_{p}^{c_{p}}\ne\hat{R}_{p}^{c'_{p}}]\le\Pr[\hat{Q}_{p}^{c_{p}}\ne\hat{R}_{p}^{c_{p}}]+\Pr[\hat{Q}_{p}^{c'_{p}}\ne\hat{R}_{p}^{c'_{p}}].
\]
Thus, considering $\{\hat{R}^{c}:c\in C\}$ as the coupling of $\{R^{c}:c\in C\}$
of the CbD approach, it follows that (\ref{eq:cbd_delta}) is less
than or equal to $\Delta$. Conversely, suppose that $\{\hat{R}^{c}:c\in C\}$
is a coupling of $\{R^{c}:c\in C\}$ with the minimum possible value
of (\ref{eq:cbd_delta}). Then, by defining $Q_{p}=\hat{R}_{p}^{c_{p}}$
and taking $(\hat{Q}^{c},\hat{R}^{c}$) with $\hat{Q}_{p}^{c_{p}}=\hat{Q}_{p}^{c'_{p}}=Q_{p}$
as the couplings of $(Q^{c},R^{c})$, we obtain
\[
\Pr[\hat{Q}_{p}^{c_{p}}\ne\hat{R}_{p}^{c_{p}}]+\Pr[\hat{Q}_{p}^{c'_{p}}\ne\hat{R}_{p}^{c'_{p}}]=\Pr[\hat{R}_{p}^{c_{p}}\ne\hat{R}_{p}^{c'_{p}}]
\]
and so $\Delta$ equals (\ref{eq:cbd_delta}). Since the minimum possible
values of $\Delta$ and $\Delta^{\textnormal{CbD}}$ over their respective
sets of considered couplings are equal, it follows that $\Delta_{0}$
and $\Delta_{0}^{\textnormal{CbD}}$ are also equal since they both
minimize the equal minimal value of $\Delta$ and $\Delta^{\textnormal{CbD}}$
over all systems whose marginals agree with those of $\{R^{c}:c\in C\}$
.
\end{proof}

\begin{thm}
\label{thm:median-for-binary}For $\pm1$-valued random variables,
we have 
\[
\Delta_{p}=\min_{\textnormal{all choices of }Q_{p}}\frac{1}{2}\sum_{c\in C_{p}}|\langle R_{p}^{c}\rangle-\langle Q_{p}\rangle|
\]
and the minimizers are characterized by
\[
\langle Q_{p}\rangle=\mathrm{Median}\{\langle R_{p}^{c}\rangle:c\in C_{p}\}.
\]
When the number of elements in $C_{p}$ is even, the median means
any value between the two middle values. If the property $p$ enters
in exactly two contexts, $C_{p}=\{c_{p},c'_{p}\}$, then $\Delta_{p}$
simplifies into
\[
\Delta_{p}=\frac{1}{2}|\langle R_{p}^{c_{p}}\rangle-\langle R_{p}^{c_{p}'}\rangle|.
\]
\end{thm}
\begin{proof}
Since the maximal coupling probability of two $\pm1$-valued random
variables is given by equality on the left side of (\ref{eq:min-pr})
in Lemma~\ref{lem:min-pr}, it can be seen that for $\pm1$-valued
random variables, (\ref{eq:Delta_p}) reduces to finding the best
$L^{1}$-approximator to the means of all $\{R_{p}^{c}:c\in C_{p}\}$
as shown in the statement of this theorem. This problem is solved
by the median.
\end{proof}

\begin{thm}
\label{thm:per-context-cbd-condition}A system is noncontexual according
to Definition~\ref{def:noncontextual-inside-measure} if and only
if there exists jointly distributed random variables $\{Q_{p}:p\in P\}$
that are minimizers of the sum in (\ref{eq:Delta_p}) and have the
property that for each $c\in C$, the system consisting of the two
bunches $S^{1}=Q^{c}$ and $S^{2}=R^{c}$ (both having the same set
of properties) is noncontexual in the CbD sense (or in the sense of
Definition~\ref{def:noncontextual-inside-measure}, since the two
are equal for systems with every property entering in precisely two
contexts). 

Furthermore, for given $\{Q_{p}:p\in P\}$, the minimum possible value
of $\Delta$ over all couplings $(\hat{R}^{c},\hat{Q}^{c})$ of $(R^{c},Q^{c})$
equals the sum over $c\in C$ of the minimum possible values of $\Delta^{\textnormal{CbD}}$
over all couplings $(\hat{S}^{1},\hat{S}^{2})$ of the subsystem $\{S^{1},S^{2}\}=\{Q^{c},R^{c}\}$
with each $\Delta^{\textnormal{CbD}}$ corresponding to the partial
sum
\[
\sum_{p\in C_{p}}\Pr[\hat{Q}_{p}^{c}\ne\hat{R}_{p}^{c}]
\]
appearing in (\ref{eq:sumdiffprob}).\end{thm}
\begin{proof}
Follows from Remark~\ref{rem:componentwise-maximality} and the definitions.
\end{proof}

\begin{thm}
\label{thm:minsum-two-properties}For a system of $\pm1$-valued random
variables, if exactly two properties $p_{c}$ and $p'_{c}$ enter
a given context $c\in C$, then the minimum possible value of the
partial sum 

\[
\sum_{p\in C_{p}}\Pr[\hat{Q}_{p}^{c}\ne\hat{R}_{p}^{c}]
\]
appearing in (\ref{eq:sumdiffprob}) over all couplings $(\hat{R}^{c},\hat{Q}^{c})$
of $(R^{c},Q^{c})$ is given by 
\[
\frac{1}{2}\max\left\{ |\langle Q_{p_{c}}^{c}Q_{p'_{c}}^{c}\rangle-\langle R_{p_{c}}^{c}R_{p'_{c}}^{c}\rangle|,\ |\langle Q_{p_{c}}^{c}\rangle-\langle R_{p_{c}}^{c}\rangle|+|\langle Q_{p'_{c}}^{c}\rangle-\langle R_{p'_{c}}^{c}\rangle|\right\} .
\]
\end{thm}
\begin{proof}
According to Theorem~\ref{thm:per-context-cbd-condition}, this problem
is equivalent to the problem of minimizing $\Delta^{\textnormal{CbD}}$
given by (\ref{eq:Delta_CbD_general}) for the system given by the
bunches $S^{1}=(Q_{p_{c}}^{c},Q_{p'_{c}}^{c})$ and $S^{2}=(R_{p_{c}}^{c},R_{p'_{c}}^{c})$
both having the same two properties identified by their position in
the bunches. This is a so called cyclic-2 system for which the solution
is known to be the above expression (see Eq.~(4) in \cite{KujalaDzhafarov2015}). 
\end{proof}

\section{Computational complexity}

In this section we illustrate the computational benefits of the present
approach using as example Alice--Bob systems with $m$ settings for
Alice and $n$ settings for Bob. 

First, consider again the Alice--Bob example with two settings for
both Alice and Bob. In the standard CbD approach, the coupling 
\[
((\hat{A}_{11},\hat{B}_{11}),(\hat{A}_{12},\hat{B}_{12}),(\hat{A}_{21},\hat{B}_{21}),(\hat{A}_{22},\hat{B}_{22}))
\]
 is represented by the $2^{8}=256$ nonnegative variables 
\[
p_{a_{11}b_{11}a_{12}b_{12}a_{21}b_{21}a_{22}b_{22}}=\Pr[\hat{A}_{ij}=a_{ij},\hat{B}_{ij}=b_{ij}:i,j\in\{1,2\}],
\]
representing the joint probabilities. For each $i,j\in\{1,2\}$, there
are 4 linear equations constraining the distribution of $(\hat{A}_{ij},\hat{B}_{ij})$
to that of the observed pair $(A_{ij},B_{ij})$. The expression (\ref{eq:cbd_delta})
can be evaluated directly from the values of these 256 variables and
so the degree of contextuality is obtained by linear programming,
minimizing this expression given the constraints (and subtracting
$\Delta_{0}^{\textnormal{CbD}}$).

In the NP approach, the negative probability joint of the hypothetical
single-indexed random variables $(A_{1},A_{2},B_{1},B_{2})$ is represented
by the 16 signed variables

\begin{equation}
p_{a_{1}a_{2}b_{1}b_{2}}=\Pr[A{}_{1}=a_{1},A_{2}=a_{2},B_{1}=b_{1},B{}_{2}=b_{2}].\label{eq:np_joint_variables}
\end{equation}
For each $i,j\in\{1,2\}$ there are 4 linear equations constraining
the marginal of $(A_{i},B_{j})$ of (\ref{eq:np_joint_variables})
to the observed joint. To minimize the total negative probability
mass, the signed variables (\ref{eq:np_joint_variables}) have to
be represented by their negative and positive parts,
\[
p_{a_{1}a_{2}b_{1}b_{2}}=p_{a_{1}a_{2}b_{1}b_{2}}^{+}-p_{a_{1}a_{2}b_{1}b_{2}}^{-},\qquad p_{a_{1}a_{2}b_{1}b_{2}}^{+},p_{a_{1}a_{2}b_{1}b_{2}}^{-}\ge0.
\]
The total negative probability mass is then obtained as the linear
expression $\sum_{a_{1},a_{2},b_{1},b_{2}\in\{+1,-1\}}p_{a_{1}a_{2}b_{1}b_{2}}^{-}$
and this expression can be minimized using linear programming given
the $2\cdot16=32$ nonnegative variables and 16 equation constraints.

In the present approach, the approximating noncontextual system is
represented by the 16 nonnegative variables
\begin{equation}
p_{a_{1}a_{2}b_{1}b_{2}}=\Pr[A{}_{1}=a_{1},A_{2}=a_{2},B_{1}=b_{1},B{}_{2}=b_{2}]\label{eq:int_joint_variables}
\end{equation}
and for each $i,j\in\{1,2\}$, the coupling $((\hat{A}_{i},\hat{B}_{j}),(\hat{A}_{ij},\hat{B}_{ij}))$
of $((A_{i},B_{j}),(A_{ij},B_{ij}))$ is represented by the 16 nonnegative
variables 
\begin{equation}
p_{a_{ij}b_{ij}a{}_{i}b_{i}}^{ij}=\Pr[\hat{A}_{ij}=a_{ij},\hat{B}_{ij}=b_{ij},\hat{A}{}_{i}=a{}_{i},\hat{B}{}_{i}=b{}_{i}].\label{eq:ext_coupling_variables}
\end{equation}
For each $i,j\in\{1,2\}$ there are 4 linear equations constraining
the marginal of $(\hat{A}_{ij},\hat{B}_{ij})$ of (\ref{eq:ext_coupling_variables})
to the observed joint and another 4 linear equations constraining
the marginal of $(\hat{A}_{i},\hat{B}_{j})$ in (\ref{eq:ext_coupling_variables})
to agree with the margial $(A_{i},B_{j})$ in (\ref{eq:int_joint_variables}).
The degree of contextuality is obtained by minimizing (\ref{eq:consum})
under these constraints. This formulation has a total of $16+4\cdot16=80$
variables with nonnegativity constraints compared to the $256$ of
the standard CbD approach and it has $4\cdot4+4\cdot4=32$ linear
equations compared to the $4\cdot4=16$ of the standard CbD apporoach.

\begin{table}
\caption{\label{tab:Linear-programming-problem}Linear programming problem
size for the Contextuality-by-Default, Negative Probabilities, and
the present approach with $m$ settings for Alice and $n$ settings
for Bob ($\pm1$ outcomes). The linear programming problem is of the
form $Mq=p$ subject to $q\ge0$, where $M$ is a matrix determined
by $m,n$ and $p$ is a vector of the probabilities of the possible
joint outcomes in each context (padded by the same number of $0$'s
in the present approach). }

\centering{}%
\begin{tabular}{|c|ccc|}
\cline{2-4} 
\multicolumn{1}{c|}{} & CbD & NP  & Present approach\tabularnewline
\hline 
Nonnegative variables (columns of $M$) & $2^{mn}$ & $2^{m+n+1}$ & $2^{m+n}+16mn$\tabularnewline
Equation constraints (rows of $M$) & $4mn$ & $4mn$ & $8mn$\tabularnewline
Inequality constraints & $0$ & $0$ & $0$\tabularnewline
\hline 
\end{tabular}
\end{table}

For a more general Alice--Bob system, with $m$ settings for Alice
and $n$ settings for Bob, the number of nonnegative variables needed
for the present approach is $2^{m+n}+16mn$ with $8mn$ equation constraints.
For the standard CbD approach, there are $2^{mn}$ nonnegative variables
with $4mn$ equation constraints and for the NP approach there are
$2^{m+n+1}$ nonnegative variables with $4mn$ equation constraints.
These numbers are summarized in Table~\ref{tab:Linear-programming-problem}.
It is clear that based on the linear programming problem size, the
present approach should be computationally more efficient than either
the standard CbD or the NP approach: already for $m,n\ge4$, the number
of nonnegative variables needed in the linear programming task is
smaller for the present approach than for either of the two previous
approaches. The number of equation constraints needed is double in
the present approach compared to either of the two previous approaches.
However, this is in fact a benefit as well, since any nonredundant
equation constraints are simply used to eliminate the same number
of variables from the system by linear program solvers.

\section{Conclusions}

We conclude by considering the logic of the present approach in an
abstract setting, with possible generalizations, and its relation
to the previously proposed Contextuality-by-Default approach.

\subsection{Approximation in abstract terms}

Let us denote by $\mathcal{S}$ the class of all systems, by $\mathcal{C}\subset\mathcal{S}$
the class of consistently connected systems, by $\mathcal{N}\subset\mathcal{C}$
the class of noncontextual systems in the traditional sense (i.e.,
determined by a single-indexed system of proper random variables),
and by $\mathcal{S}_{R}\subset\mathcal{S}$ the class of systems $R'\in\mathcal{S}$
whose marginals agree with $R$ (i.e., ${R'}{}_{p}^{c}\sim R_{p}^{c}$
for all $c\in C$, $p\in P_{c}$). In abstract terms, the main idea
of the present approach is to model a (possibly) inconsistently connected
observable system $R\in\mathcal{S}$ by a consistently connected system
$Q$ that approximates it as closely as possible within a given subclass
$\mathcal{C}_{0}\subset\mathcal{C}$ of consistently connected systems.
A certain computationally convenient metric\footnote{It is easy to show that $\Delta$ given by (\ref{eq:sumdiffprob})
and Definition~\ref{def:general-approximation} for couplings $(\hat{R}^{c},\hat{Q}^{c})$
of $(R^{c},Q^{c})$ that minimize (\ref{eq:sumdiffprob}) is indeed
a metric when taken as a function of the two sets of bunches $R=\{R^{c}:c\in C\}$
and $Q=\{Q^{c}:c\in C\}$.} $\Delta(R,Q)$ is defined to measure in a probabilistic sense how
far the observed system $R$ is from its closest approximation. If
the minimum value of $\Delta(R,Q)$ over all $Q\in\mathcal{C}_{0}$
equals $\Delta_{0}(R)$, the minimum value of $\Delta(R',Q)$ over
all $Q\in\mathcal{C}$ and all $R'\in\mathcal{S}_{R}$, then the approximation
is said to be optimal. 

If the approximation is optimal and the approximating system $Q$
is a noncontextual system in the traditional sense (i.e., $Q\in\mathcal{N})$,
then $R$ is considered noncontextual (apart from the trivial cross-influences
that cause the inconsistent connectedness). If optimal approximation
by $Q\in\mathcal{N}$ is not possible, then contextuality can be measured
either as the minimum possible value of $\Delta(R,Q)-\Delta_{0}(R)$
over all approximating $Q\in\mathcal{C}_{0}=\mathcal{N}$ or one can
consider all (contextual or not) optimally approximating $Q\in\mathcal{C}_{0}=\mathcal{C}$
and apply and minimize any measure of contextuality on all such $Q$.
Since $Q$ is consistently connected, the latter formulation allows
generalizing any measure of contextuality that is defined for consistently
connected systems to inconsistently systems. As an example, we have
done this for the previously defined negative probability (NP) measure
of contextuality (see Theorem~\ref{thm:equivalence-with-NP}). Thus,
a measure of contextuality can be considered on the ``outside''
as the minimal distance from an approximating noncontextual system
or ``inside'', by applying and minimizing an existing measure of
contextuality over all optimally approximating consistently connected
systems. 

For now, we have defined the concept of optimal approximation so that
the distance $\Delta(R,Q)$ can be positive only by the amount $\Delta_{0}(R)$
required by the inconsistency of the marginal distributions of $R_{p}^{c}$,
$c\in C$, $p\in P_{c}$. However, if the considered class $\mathcal{C}_{0}$
of models does not include all possible consistently connected marginals,
one possible generalization is to change the definition of the threshold
$\Delta_{0}(R)$ so that $\Delta(R',Q)$ is minimized over $R'\in\mathcal{S}_{R}$
and $Q\in\mathcal{C}_{0}$ (as opposed to $R'\in\mathcal{S}_{R}$
and $Q\in\text{\ensuremath{\mathcal{C}}}$). This would amount to
testing if there is contextuality on top of that predicted by the
considered class $\mathcal{C}_{0}$ of approximating systems and on
top of any changes of the marginals from those allowed by the class
$\mathcal{C}_{0}$ of approximating systems. In Example~\ref{ex:EPR},
for example, $\mathcal{C}_{0}$ is defined by a single model with
uniform marginals and so our standard definition would never consider
an observed system noncontextual if all the observed marginal expectations
of some property deviate from the predicted 0 to the same direction
(this is because optimal approximation would in that case require
the marginal expectations of $Q$ to deviate to the same direction
which is not allowed by the predicted uniform marginals of $Q\in\mathcal{C}_{0}$).
However, $\Delta_{0}(R',Q)$ optimized over $R'\in\mathcal{S}_{R}$
and $Q\in\mathcal{C}_{0}$ would yield a threshold that allows detecting
if the distance of the approximating system $Q\in C_{0}$ is not larger
than the minimum required by the changes in the marginals from those
allowed by $\mathcal{C}_{0}$.

\subsection{Relation to Contextuality-by-Default}

The main difference of the present approach to CbD is that the present
approach does not consider a coupling of \emph{all} random variables
$\{R_{p}^{c}:c\in C,\ p\in P\}$ but only refers to couplings of one
bunch (with the corresponding bunch of the approximating system) at
a time in the definition of $\Delta$ (and subcouplings of those in
the definition of $\Delta_{0}$). As shown by proofs analogous to
those used above, given the coupling $\{\hat{R}^{c}:c\in C\}$ of
$\{R^{c}:c\in C\}$ of the CbD approach, one can always make an arbitrary
choice of $c_{p}\in C_{p}$ for all $p\in P$ to define a single-indexed
system $Q_{p}=\hat{R}_{p}^{c_{p}}$, $p\in P$, that approximates
in some sense the observed system. For systems with each property
entering in precisely two contexts, this single-indexed system approximates
optimally (in the sense of the present approach) the observed system
if and only if the subcouplings of $\{\hat{R}^{c}:c\in C\}$ corresponding
to connections are maximal, that is, if the system is noncontextual
in the CbD sense (see Theorem~\ref{thm:equivalence-with-CbD}). This
allows in particular all the analytic results of the so-called cyclic
systems \cite{KujalaDzhafarovLarsson2015,KujalaDzhafarov2015,DKL2015FooP}
derived in the CbD approach to be used for the present approach. Conversely,
due to the equivalence, the computational advantages of the present
system can be made use of to replace the calculations of the CbD approach
in all systems with each property entering in precisely two contexts. 

However, the present approach is not consistent with CbD in general.
That is, there are systems that are contextual according to one definition
but not according to the other. Still, the general idea of approximating
optimally an inconsistely connected system with a consistently connected
one can be formulated in a way that is consistent with CbD. For completeness,
we will outline such a formulation here.

Given a coupling $\{\hat{R}^{c}:c\in C\}\cup\{\hat{Q}_{p}:p\in P\}$
of the bunches $\{R^{c}:c\in C\}$ and jointly distributed random
variables $\{Q_{p}:p\in P\}$, denote

\begin{equation}
\Delta^{*}=\sum_{p\in P}\Pr[\hat{Q}_{p}\ne\hat{R}_{p}^{c}\textnormal{ for some }c\in C_{p}]\label{eq:sumdiffprob-1}
\end{equation}
and
\begin{equation}
\Delta_{0}^{*}=\sum_{p\in P}\Delta_{p}^{*},\label{eq:Delta_0-1}
\end{equation}
where
\begin{equation}
\Delta_{p}^{*}=\min_{\substack{\text{all choices of }Q_{p}}
}\min_{\substack{\text{all couplings }(T,\{S^{c}:c\in C_{p}\})\\
\textnormal{of }(Q_{p},\{R_{p}^{c}:c\in C_{p}\})
}
}\Pr[T\ne S^{c}\textnormal{ for some }c\in C_{p}].\label{eq:Delta_p-1}
\end{equation}
A system is considered contextual if $\Delta^{*}=\Delta_{0}^{*}$
and a measure of contextuality is given by $\Delta^{*}-\Delta_{0}^{*}$. 

Using analogous proofs as above, it can be shown that this condition
for contextuality is equivalent to that of the CbD approach and that
the resulting measure of contextuality is equivalent to the one we
have used for the CbD approach. However, being compatible with the
CbD approach means that we have to consider a coupling of \emph{all}
variables and so the metric $\Delta^{*}$ is not a function of the
bunches of the original and approximating system but its arguments
are full couplings of all random variables of a system and additionally
the two couplings given as arguments must be defined on the same sample
space. This means in particular that the computational benefits resulting
from the simpler, single-indexed representation will be lost.

Conversely, one can also modify the CbD approach so as to be equivalent
with the standard definition of the present approach. This would amount
to changing the requirement of the maximality of each subcoupling
$\{\hat{R}_{p}^{c}:c\in C_{p}\}$ corresponding to a connection $\{R_{p}^{c}:c\in C_{p}\}$,
$p\in P$, to requiring that there exists a random variable $Q_{p}$
satisfying
\[
\sum_{c\in C_{p}}\Pr[Q_{p}\ne\hat{R}_{p}^{c}]=\Delta_{p}.
\]
Then, the random variables $\{Q_{p}:p\in P\}$ would define an optimally
approximating system satisfying the present definition of noncontextuality.

The remaining question is whether the definition of contextuality
of the CbD approach or the definition of the present approach is more
useful or if one may be more useful in some contexts and the other
in some other contexts. An obvious example of difference between the
two definitions is given when for each property $p\in P$, the marginal
distributions of its observations in two different contexts have disjoint
supports. In that case, the system is always noncontextual according
to the CbD definition, since all subcouplings corresponding to the
connections are maximal. However, according to the present approach,
such a system may be contextual. We will construct one such example.

Suppose the system consists of four bunches $R^{1},\dots,R^{4}$,
each measuring the same two $\pm1$-valued properties $p\in\{1,2\}$.
Suppose $R_{1}^{1}=R_{2}^{1}$ and $R_{1}^{2}=-R_{2}^{2}$ are both
uniformly distributed (so that $\langle R_{1}^{1}\rangle=\langle R_{2}^{1}\rangle=\langle R_{1}^{2}\rangle=\langle R_{2}^{2}\rangle=0$,
$\langle R_{1}^{1}R_{2}^{1}\rangle=1$, and $\langle R_{1}^{2}R_{2}^{2}\rangle=-1$)
and $R_{1}^{3}=R_{2}^{3}=1$ and $R_{1}^{4}=R_{2}^{4}=-1$ (so that
$\langle R_{1}^{3}\rangle=\langle R_{2}^{3}\rangle=1$, $\langle R_{1}^{4}\rangle=\langle R_{2}^{4}\rangle=-1$,
and $\langle R_{1}^{3}R_{2}^{3}\rangle=\langle R_{1}^{4}R_{2}^{4}\rangle=1$).
Then, $R_{p}^{3}$ and $R_{p}^{4}$ have disjoint supports for $p\in\{1,2\}$
and so the system is noncontextual in the CbD sense. Also, the median
of $\left\{ \langle R_{p}^{c}\rangle:c\in\{1,2,3,4\}\right\} =\{0,0,1,-1\}$
is $0$ for $p\in\{1,2\}$ and so, according to Theorem~\ref{thm:median-for-binary},
$\langle Q_{p}\rangle=0$ for $p\in\{1,2\}$ is a necessary condition
for $\Delta=\Delta_{0}$. Theorem~\ref{thm:median-for-binary} also
yields the value $\Delta_{0}=2$ from the differences of marginals:
$\Delta_{p}=\frac{1}{2}(|0-0|+|0-0|+|1-0|+|-1-0|)=1$ for $p=1,2$.
Now, according to Theorem~\ref{thm:minsum-two-properties} the minimum
possible value of the partial sum
\[
\sum_{p\in\{1,2\}}\Pr[\hat{Q}_{p}^{c}\ne\hat{R}_{p}^{c}]
\]
appearing in the definition of $\Delta$ is given by
\[
\frac{1}{2}\max\left\{ |\langle Q_{1}Q_{2}\rangle-\langle R_{1}^{c}R_{2}^{c}\rangle|,\ |\langle Q_{1}\rangle-\langle R_{1}^{c}\rangle|+|\langle Q_{2}\rangle-\langle R_{2}^{c}\rangle|\right\} .
\]
The sum of this expression over $c\in\{1,2,3,4\}$ is
\begin{align*}
 & \frac{1}{2}\max\left\{ |\rho-1|,\ |0-0|+|0-0|\right\} \\
 & +\frac{1}{2}\max\left\{ |\rho-(-1)|,\ |0-0|+|0-0|\right\} \\
 & +\frac{1}{2}\max\left\{ |\rho-1|,\ |0-1|+|0-1|\right\} \\
 & +\frac{1}{2}\max\left\{ |\rho-1|,\ |0-(-1)|+|0-(-1)|\right\} \\
 & =\frac{1}{2}(1-\rho)+\frac{1}{2}(1+\rho)+1+1=3>2=\Delta_{0}
\end{align*}
for all $\rho\in[-1,1]$ and so the system is contextual according
to the present approach.

\section*{Acknowledgements}

This author is grateful to Ehtibar Dzhafarov, Acacio de Barros, and
Gary Oas for discussions related to this work and to Ehtibar Dzhafarov
for many helpful comments.

\bibliographystyle{plain}

\end{document}